\documentclass[11pt,reqno]{amsart}
\usepackage{mathrsfs, amsmath,amssymb,amsfonts, amsthm, dsfont}
\usepackage{float}
\restylefloat{table}
\usepackage{hyperref}
\usepackage{epsfig}
\usepackage{amscd}
\usepackage{amsmath, amssymb, amsthm}
\usepackage{graphics}
\usepackage{color}
\usepackage{dsfont}

\newtheorem{theorem}{Theorem}[section]

\numberwithin{equation}{section}

\newcommand\II{\text{II} }

\renewcommand{\Pi}{\varPi}

\renewcommand{\epsilon}{\varepsilon}

\numberwithin{equation}{section}

\newtheorem{theo}{{\sc Theorem}}
\newtheorem{cor}[theo]{{\sc Corollary}}

\newtheorem{lem}[theo]{{\sc Lemma}}

\newtheorem{prop}[theo]{{\sc Proposition}}

\newenvironment{rem}{\medskip\noindent{\it Remark:\/} }{\medskip}

\newcommand{\R}{{\mathbb R}}

\newtheorem*{main-theorem}{Main Theorem}
\newtheorem*{old-thm}{Theorem}
\theoremstyle{definition}
\numberwithin{equation}{section}

\def\11{\mathds{1}}

\def\Ci{{\mathcal C}^\infty}

\def\phi{\varphi}
\def\half{{\frac{1}{2}}}

\def\be{\begin{eqnarray*}}
\def\ee{\end{eqnarray*}}
\def\ben{\begin{eqnarray}}
\def\een{\end{eqnarray}}

\def\lll{\left\langle}
\def\rrr{\right\rangle}

\def\L2R{L_{\text{Rest}}^2}

\def\tchi{\tilde{\chi}}

\newcommand{\ocal}{\mathcal{O}}

\newcommand{\scal}{\mathcal{S}}

\begin{document}
\title[Logarithmic lower bound]{Logarithmic lower bound on the number of nodal domains}


\author{Steve Zelditch}
\address{Department of Mathematics, Northwestern  University, Evanston, IL 60208, USA}
\email{zelditch@math.northwestern.edu}

\thanks{Research partially supported by NSF grant DMS-1541126  }

\begin{abstract}

We prove that the number of nodal domains of eigenfunctions grows at least
logarithmically with the eigenvalue (for almost the entire sequence of eigenvalues) on certain negatively curved surfaces. The geometric model is the same as in prior joint work with J. Jung, where  the number of nodal domains was shown to tend to infinity.
 The
surfaces are assumed to be ``real Riemann surfaces'', i.e. Riemann surfaces
with an  anti-holomorphic involution $\sigma$ with non-empty fixed point set. The
eigenfunctions are assumed to be even or odd, which is automatically the case for
generic invariant metrics. The logarithmic growth rate gives a quantitative refinement of the prior results.

\end{abstract}
\maketitle

\section{Introduction}

In recent articles \cite{JZ,JZ2} (see also \cite{JS})  J. Jung and the author proved that for certain
non-positively curved surfaces, the number of nodal domains of an orthonormal
basis $\{u_j\}$ of Laplace  eigenfunctions tends to infinity with the eigenvalue along almost the entire sequence
of eigenvalues. This nodal counting result built on prior  work of Ghosh-Reznikov-Sarnak
\cite{GRS} in the case of the modular domain, which gave a power law lower bound
on the number of nodal domains for individual Maass-Hecke eigenfunctions. Their proof uses
methods of L-functions and assumes a certain  Lindelof hypothesis, while those
of \cite{JZ,JZ2} use PDE methods to obtain  unconditional results for a density
one subsequence of eigenfunctions. In \cite{JS},
Jang-Jung  used a clever Bochner positivity argument  to obtain unconditional results for individual Maass-Hecke eigenfunctions of arithmetic triangle groups.  However, no
growth rate was specified in \cite{JZ,JZ2} or in \cite{JS}\footnote{There are two independent difficulties in the lower bounds: (i) Obtaining a quantitative lower bound for a density
one subsequence, (ii) Obtaining even a qualitative  uncondiitional lower bound for the entire sequence, i.e. ``for 
individual eigenfunctions.'' }
 The main result of this note (Theorem \ref{theo1}) improves
the qualitative result of \cite{JZ}  to a quantitative  logarithmic lower bound
for the number of nodal domains of a density one subsequence of even/odd eigenfunctions of the Laplacian on 
surfaces of negative curvature possessing an isometric involution.  The structure
of the proof is the same as in \cite{JZ} but two key estimates are sharpened. The main new input is the logarithmic quantum ergodic result of Hezari-Rivi\`ere \cite{HeR} and X. Han \cite{H}. At the present time,  a logarithmic growth rate is the best that can be expected  due to the exponential growth rate of the geodesic flow in negative curvature
and its impact on all remainder estimates and localization estimates on the spectrum. 
Before stating the results, we recall some background and terminology from \cite{JZ}.

Although the main result occurs in dimension 2, we start with some general
notational conventions in any dimension. Let $(M,g)$ be a  compact $C^{\infty}$ $d$-dimensional manifold. We denote the Laplacian of $g$ by $\Delta$ and state the
eigenvalue problem as  
$$\Delta u_{\lambda} = - \lambda \; u_{\lambda}. $$ 
Eigenfunctions are always assumed to be $L^2$-normalized, $||u_{\lambda}||^2
= \int_M |u_{\lambda}|^2 dV_g = 1$, where $dV_g$ is the Riemannian volume form.
We often fix an orthonormal basis $\{ u_j\}_{j = 1}^{\infty}$ with $\lambda_0 = 0 < \lambda_1 \leq \lambda_2 \cdots$. 

The nodal set of  an  eigenfunction  of the Laplacian
is denoted by
$$
Z_{u_{\lambda}} = \{x: u_{\lambda}(x) = 0\}.
$$
The key object in this note is the  number  $N(u_{\lambda} )$  of nodal domains of $\phi_{\lambda}$, i.e. the number of connected components
$\Omega_j$ of the complement of the nodal set,
$$M \backslash Z_{u_{\lambda}} =  \bigcup_{j = 1}^{N(u_{\lambda} )} \Omega_j. $$

We now restrict to setting of 
 \cite{JZ}, in which  $M$ is assumed to be a Riemann surface of genus $\mathfrak{g}$   (with complex structure $J$)
possessing an anti-holomorphic
involution $\sigma$   whose fixed point set $\mathrm{Fix}(\sigma)$  is non-empty.\footnote{In \cite{JZ}, the fixed point set is assumed to be separating but this assumption is not necessary.}  
 We define
$\mathcal{M}_{M, J, \sigma}$ to be the space  of  $C^{\infty}$  $\sigma$-invariant {\it negatively curved} Riemannian metrics on a real Riemann surface
$(M, J, \sigma)$.  As discussed in \cite{JZ}, 
$\mathcal{M}_{M, J, \sigma}$ is an open set in the space of $\sigma$-invariant metrics, and in particular is infinite dimensional.
For each $g \in {\mathcal M}_{M, J, \sigma}$,  the fixed point set $\mathrm{Fix}(\sigma)$
is a disjoint union
\begin{equation} \label{H}   \mathrm{Fix}(\sigma)  = \gamma_1 \cup \cdots \cup \gamma_n \end{equation} of $0 \leq n \leq \mathfrak{g} + 1$ simple closed geodesics.

The isometry $\sigma$ acts  on $L^2(M, dA_g)$, and we define
$L^2_{even}(M)$, resp. $L^2_{odd}(M)$,  to denote the subspace of even
functions $f(\sigma x) = f(x)$, resp. odd elements $f(\sigma x) = - f(x)$.
 Translation by any isometry
$\sigma$ commutes with the Laplacian $\Delta_g$ and so  the
even and odd parts of eigenfunctions  are eigenfunctions, and all eigenfunctions
are linear combinations of   even or odd eigenfunctions. We denote by $\{\phi_j\}$  an orthonormal basis of   $L_{even}^2(M)$ of  even eigenfunctions, resp. $\{\psi_j\}$ an orthonormal basis   of $L_{odd}^2(M)$ of odd eigenfunctions, with respect to the inner product
$\langle u, v \rangle = \int_M u \bar{v} dV_g$, ordered by the
corresponding sequence of eigenvalues $\lambda_0 = 0 < \lambda_1 \leq \lambda_2 \uparrow \infty$.  We write $N(\phi_j) $ for $N(\phi_{\lambda_j} )$.
In \cite{JZ} we proved  that for
generic metrics in $\mathcal{M}_{M,J, \sigma}$, the eigenvalues are simple (multiplicity one) and therefore all eigenfunctions are either
even or odd.

The main result of this note concerns the case $d =2$:

\begin{theorem}\label{theo1}
Let $(M,J, \sigma)$ be a  compact real Riemann surface of genus $\mathfrak{g} \geq 2$ with anti-holomorphic involution $\sigma$ satisfying  $Fix(\sigma)\neq \emptyset$. Let $\mathcal{M}_{M, J, \sigma}$
be the space of $\sigma$-invariant negatively curved $C^\infty$ Riemannian metrics on $M$.Then for
any $g \in \mathcal{M}_{(M, J, \sigma)}$ and  any orthonormal $\Delta_g$-eigenbasis $\{\phi_j\}$ of $L_{even}^2(M)$, resp. $\{\psi_j\}$ of $L_{odd}^2(M)$, one can find a density $1$ subset $A$ of $\mathbb{N}$ and a constant $C_g > 0$ depending only on $g$ such that, for $ j \in A$
\[N(\phi_j)  \geq 
 C_g\; (\log \lambda_j)^{K}, \;\;  (\forall K < \frac{1}{6}).
\]
resp.
\[
N(\psi_j) \geq C_g \;
 (\log \lambda_j)^{K} , \;\; (\forall K < \frac{1}{6}).
\]


\end{theorem}

\begin{rem} The constraint on $K$ is the one in the logarithmic scale QE (quantum ergodic) results of \cite{HeR} and \cite{H}. In \cite{HeR}, the authors used higher variance moments
to improve the constraint to $K < \frac{1}{2d}$. It should be possible to improve Theorem \ref{theo1} in the same way, but at the 
expense of additional technicalities that seem out of proportion to the improvement. The main point is that the arguments of \cite{JZ} lead to quantitative
estimates, and the point seems well enough established with the smaller value of $K$. It is not yet clear what is the threshold for
small scale quantum ergodicity. An improvement from the logarithmic scale to a power scale would imply a similar improvement
for the nodal count.

\end{rem}

\begin{rem} 

In \cite{JZ2} the authors proved a much more general qualitative result stating that the number of nodal tends tends to infinity for
surfaces of non-positive curvature and concave boundary.  As explained in a later remark, there are several obstructions to
generalizing the logarithmic lower bound to such surfaces, although it should eventually be possible to over-come them.

\end{rem}

\subsection{Notations for eigenvalues and logarithmic parameters}
To maintain notational consistency with \cite{ctz,HeR,H} we also denote sequences
of eigenfunctions by the semi-classical notation $u_{h_j}$ or just  $u_h$ with $h = h_j= \lambda_j^{-\half}$;
equivalently, we fix $E$ and put
$\lambda_j = h_j^{-2}E$ (as in \cite{H, HeR}).
 Because of the homogeneity of
the eigenvalue problem, there is no loss of generality in setting $E =1 $,
and then we consider eigenvalues  $E_j \in [1, 1 + h]$, or in homogeneous notation
$\sqrt{\lambda_j} \in [h^{-1}, h^{-1} + 1] = [\sqrt{\lambda}, \sqrt{\lambda} + 1]$
We denote by $N(\lambda) \simeq C_d Vol(M, g)  \lambda^{\frac{d-1}{2}} $ the number of eigenvalues in
the interval $[\sqrt{\lambda}, \sqrt{\lambda} + 1]$.\footnote{ Eigenvalues are denoted by $\lambda^2$ in \cite{ctz}
and by $\lambda$ here and in \cite{JZ}.}

We further introduce a  logarithmically
 small parameter for a manifold $M^d$ of dimension $d$
  \begin{equation} \label{ell} \ell_j =  |\log \hbar|^{-K} = (\log \lambda_j)^{- K}  \;\rm{ where}\; 0 < K < \frac{1}{3d}. \end{equation}
Han  \cite{H} uses the notation $\alpha = K$ and denotes the same quantity by    \begin{equation} \label{deltah} \delta(\hbar) = |\log \hbar|^{-\alpha}=  r(\lambda_j), \;\;\alpha < \frac{1}{3d}, \end{equation}  
We adopt both notations.
Other choices of $\delta(h)$ may arise in applications and will be specified below.

\subsection{Main new steps of the proof}

The main step in the proof that $N(u_j) \to \infty$ in \cite{JZ,JZ2} was to prove that
for any smooth connected arc $\beta \subset \rm{Fix}(\sigma)$, $u_j |_{\beta}$ has a sign-changing
zero in $\beta$. To obtain logarithmic lower bounds, we need to prove the
existence of a sign-changing zero on   sequences $\beta_j$ of shrinking arcs
with lengths $|\beta_j| = \ell_j$ \eqref{ell}. More precisely,  we partition $\rm{Fix}(\sigma)$ into $\ell_j^{-1}$  open intervals of lengths $\ell_j$  and show that $u_j$
has a sign changing zero in each interval.

The quantitative improvements apply to general smooth hypersurfaces $H \subset M$
of general Riemannian manifolds of any dimension $d$. Later we specialize the
results to the surfaces in Theorem \ref{theo1} in dimension $d = 2$ and with
$H = \rm{Fix}(\sigma). $
In general dimensions, the  partitions are defined by choosing a cover of $H$ by  $C \ell^{-1}$  balls   of radius $\ell$ with centers $\{x_k\} \subset H $ at a net of points
of $H$ so that  
\begin{equation} \label{COVER} H \subset \bigcup_{k = 1}^{R(\ell)} B(x_k, C \ell ) \cap H. \end{equation}
Here and hereafter, $C$ or $C_g$ denotes a positive constant depending only on $(M, g, H)$
and not on $\lambda_j$.
The cover may be constructed so that each point of $H$ is contained in at most $C_g$ of the double
balls $B(x_k, 2 \ell)$. The number of such balls satisfies the bounds,
$$c_1 \ell^{-d +1} \leq R(\epsilon) \leq C_2 \ell^{-d +1}. $$

There were  three analytic ingredients in the proof of existence of a sign changing zero in every arc $\beta$ in  \cite{JZ}. Two of them  need to be improved to give logarithmic lower bounds.

\begin{itemize}

\item (i)\; One needs to prove a QER  (quantum ergodic restriction) theorem in the sense of \cite{ctz} on the length
scale $O(\ell_j)$, which says (roughly speaking) that there exists a subsequence of eigenfunctions $u_{j_n}$ of density one so that 
matrix elements of the restricted eigenfunctions tend to their Liouville limits simultaneously for all the covering balls of \eqref{COVER}. 
Since there are $(\log \lambda)^K$ such balls, the scale of the  QER theorem is constrained by \eqref{ell}. To be more precise, 
we only need a weaker result giving  lower bounds rather than asymptotics, as stated in Proposition \ref{QERh}. \bigskip

\item (ii)\; One needs to prove a small scale Kuznecov asymptotic formula in the sense of \cite{z,HHHZ}, to the effect that there exists a
subsquence of density one for which $\int_{\beta} u_{j_k}$ is of order  $|\beta|\lambda_j^{-\frac{1}{4}} (\log \lambda_j)^{1/3}$ when $|\beta| \simeq \ell_j.$\footnote{The power $\frac{1}{3}$ in $ (\log \lambda_j)^{1/3}$ is  chosen
for later convenience. It could be any power $< \frac{1}{2}$.}  Again, one needs to show that there is a subsequence of density one for which this estimate   holds simultaneously for all the balls of the cover. 

\bigskip

\item (iii) \;  The sup-norm estimate $||u_j||_{\infty} = O(\frac{\lambda_j^{\frac{1}{4}}}{\sqrt{\log \lambda_j}})$ of B\'erard \cite{Be} used in \cite{JZ} does not need to be modified.

\end{itemize}
\bigskip

For background on QE (quantum ergodicity) theorems we refer to \cite{Zw} (and to the origins, \cite{Schnirelman}). We assume here the reader's familiarity with the
basic notions and with the QER (quantum ergodic restriction) problem (see \cite{ctz}).

\begin{rem}  In order to obtain the  logarithmic lower bound on nodal domains, one needs to prove the QER  theorem and the Kuznecov bound on shrinking balls  with the same radius $\ell_j$. In fact, one can shrink intervals much more in
the Kuznecov bound since the estimates involve the principal term in an asymptotic expansion rather than the remainder term. 
This remark will be explained at the end of the proof of Proposition \ref{KUZ}.  \end{rem} 

A logarithmic scale QE theorem asserts, roughly speaking, that matrix elements $\langle Op_h(a_{\ell}) u_{j_k}, u_{j_k} \rangle$
of eigenfunctions with respect to logarithmically dilated symbols $a_{\ell}(x, \xi) = a(\frac{x - x_0}{\ell}, \frac{\xi - \xi_0}{\ell})$ are asymptotic to their Liouville averages
$\int_{T^*M} a_{\ell} d\mu_L$ (see \cite{H} for the precise formulation of the symbols). For  nodal domain counting, as
for nodal bounds in \cite{HeR} it is crucial to have some kind of uniformity of the limits as the base point $(x_0, \xi_0)$ varies. 
The simplest version would be that the QE limits are uniform in the symbol and the base point, but such a uniform result is
lacking at this time. Instead there exist uniform upper and lower bounds on the mass of the eigenfunctions in all of the logarithmically
shrinking balls of the cover \eqref{COVER} (see Theorem \ref{HeRres} and Theorem \ref{HANCOR}). We will adapt these bouds to the QER problem.


The small scale QER statement is the following `uniform comparability result', based on Proposition 3.1 of \cite{HeR}  
(see Proposition \ref{HeRres}) and Corollary 1.9 of \cite{H}.

\begin{prop}\label{QERh} Let $(M, g)$ be a compact negatively curved
manifold of dimension $d$  without boundary,and let $H$ be a smooth hypersurface.  Let $\ell$ be defined by \eqref{ell}. Then for any orthonormal basis 
of eigenfunctions $\{u_j\}$, and $K$ as in \eqref{ell},  there exists a full density subsequence $\Lambda_K$ of ${\mathbb N}$ so that for $j \in \Lambda_K$  and for every
$1 \leq k \leq R(\ell), $ and centers $x_k$ of \eqref{COVER},
$$\int_{B(x_k, C \ell ) \cap H}\left( |u_j|^2  +  |\lambda_j^{-\half}  \partial_{\nu} u_j|^2 \right) dS_g   \geq a_1 \ell^{d-1}  = a_1 (\log \lambda_j)^{- (d - 1) K}$$
where $a_1, a_2 > 0$ depend only on $\chi, g$ and $K$ is defined in \eqref{ell}. \end{prop}
Here, and hereafter, $\partial_{\nu}$ denotes a fixed choice of normal derivative along $H$.
We apply the result when $\dim M = 2$ and $H = \rm{Fix}(\sigma)$, and the
eigenfunctions are even or odd. In that case, one of the two terms above drops
and we get,

\begin{cor} \label{BIGintro} Let $(M, J, \sigma, g)$ be a negatively curved surface with isometric
involution. Then  for any orthonormal basis of even eigenfunctions $\{\phi_j\}$, resp.
odd eigenunctions $\{\psi_j\}$,  there exists a full density subsequence $\Lambda_K$  so that for $j_n \in \Lambda_K$,  

$$\int_{B(x_k, C \ell ) \cap H}  |\phi_{j_n}|^2   dS_g   \geq  a_1 (\log \lambda_{j_n})^{-  K}$$
and
$$\int_{B(x_k, C \ell ) \cap H}  |\lambda_j^{-\half}  \partial_{\nu} \psi_{j_n}|^2 dS_g   \geq a_1  (\log \lambda_{j_n})^{-  K}$$
uniformly in $k$.

\end{cor}

 The statement is termed
a uniform comparability result (by Han and Hezari-Rivi\`ere) because it does not assert uniform
convergence of the sequences as $k$ varies but only gives uniform upper and lower bounds. The loss of asymptotics is just due to the passage from smooth test functions to characteristic functions of balls. We only state the lower
bound because it is the one which is relevant for nodal counts.

The next step is to prove a uniform logarithmic scale Kuznecov period bound in 
the sense of \cite{z,HHHZ}. It is also
a general result, but for the
sake of simplicity, and because it is the case relevant to this note, we assume
$\dim M = 2$. 
\begin{prop} \label{KUZ} In the notation of Proposition \ref{QERh},  let $H =  \rm{Fix}(\sigma)$, let $K$ be as in \eqref{ell} and $\{x_k\}$ the centers of \eqref{COVER}.  Then for a subsequence of $\Lambda_K \subset {\mathbb N}$ of density one, if $j_n \in \Lambda_K$,

$$ \left| \int_{B(x_k, C \ell) \cap H}  \phi_{j_n} ds \right|  \leq C_0  \ell_j \; \lambda_j^{-\frac{1}{4}} \;(\log \lambda_j)^{1/3}= C_0 (\log \lambda_j)^{- K} \lambda_j^{-\frac{1}{4}} \;(\log \lambda_j)^{1/3} ,$$ 
resp.
$$ \left| \int_{B(x_k, C \ell) \cap H}  \lambda_{j_k}^{-\frac{1}{2}} \partial_{\nu} \psi_{j_n} ds \right|  \leq C_0  \ell_j \; \lambda_j^{-\frac{1}{4}} \;(\log \lambda_j)^{1/3} = C_0 (\log \lambda_j)^{- K} \lambda_j^{-\frac{1}{4}} \;(\log \lambda_j)^{1/3} ,$$ 
uniformly in $k$.
\end{prop}

\begin{rem} As mentioned in a remark above, one could improve the result by letting $\ell = \lambda^{-\epsilon}$ for suitable
$\epsilon$.  But the improvement is not useful for nodal counts until (or if) one can improve the small-scale
quantum ergodicity result to scales of the form $\lambda^{-\epsilon}$; such a bound
would open the possibility of improving the nodal count
to a power of $\lambda$. \end{rem}

Granted the Propositions, the proof of existence of a sign-changing zero
is the same as in \cite{JZ,JZ2}. We give a sketch here to orient the reader,
with fuller details in \S \ref{CONCLUSION}.  Combining the sup-norm estimate (iii) and the QER lower bound (i) of Proposition \ref{QERh}, there exists $C > 0$ so that   $$C \ell_j < \int_{B(x_j, C \ell) \cap H} ||\phi||^2 \leq \frac{\lambda_j^{\frac{1}{4}}}{\sqrt{\log \lambda} } \int_{B(x_j, C \ell) \cap H} |\phi_j|.$$  follows that
$$ \int_{B(x_j, C \ell) \cap H} |\phi_j| \geq \lambda_j^{-\frac{1}{4}} (\log \lambda_j)^{\half}\;\ell_j. $$
But by the Kuznecov upper bound (iii) of Proposition \ref{KUZ},
$$|\int_{B(x_j, C \ell) \cap H} \phi_j |\leq \lambda_j^{-\frac{1}{4}} \ell_j\; (\log \lambda_j)^{1/3}.  $$
Hence,
$$\left| \int_{B(x_j, C \ell) \cap H} \phi_j ds \right| < \int_{B(x_j, C \ell) \cap H}| \phi_j| ds,$$
and thus $\phi_j |_{B(x_j, C \ell) \cap H}$ has a sign changing zero. 

The remainder of the argument is identical to that of \cite{JZ, JZ2}.
For the rest of this note we only discuss Propositions \ref{QERh} and  \ref{KUZ}.

\begin{rem} We use the same scale $\ell_j$ in both the quantum ergodic restriction result and the Kuznecov 
formula.As a result, it cancels out from the inequalities. We could use a smaller scale in the Kuznecov bound. But the size of $|\beta|$ is constrained by the more delicate QER result,
so there is no gain if we shrink the interval in the Kuznecov formula.
\end{rem} 

\begin{rem}
As mentioned above,  there are two aspects
of the proof that are difficult to generalize to surfaces of negative curvature and concave boundary.  First, the logarithmic QE result of \cite{HeR,H} is at present only proved in the boundaryless case. Second, the sup norm estimate (iii)
has not been proved at this time for negatively curved surfaces with concave boundary.  We conjecture that both obstacles can
be overcome and that the logarithmic growth rate of nodal domains holds in the setting of \cite{JZ2}.

\end{rem}

\subsection{Acknowledgements}

We thank H. Hezari,  G. Rivi\`ere and X. Han for comments on earlier versions of this article, in particular on uniformity of the log scale
QE results as the centers of the balls of the cover \eqref{COVER} vary. We also
thank J. Jung for correcting some typos. 

\section{Proof of Proposition \ref{QERh}}

The purpose of this section is to prove the uniform lower bound of Proposition \ref{QERh}. We use the Rellich identity argument
of \cite{ctz} to extract a logarithmic scale QER (quantum ergodic restriction) result from the global ones of \cite{HeR} and \cite{H}. However, as mentioned above, a key
new issue is to obtain uniformity in the centers $\{x_k\}$ of the covering balls \eqref{COVER}. In principle, one would like to 
prove existence of a subsequence of eigenfunctions of density one for which one has  Liouville weak* limits with uniform remainders in all the balls $B(x_k, C \ell) \cap H$,
but this has not yet been established globally in \cite{HeR, H}. For the proof of Theorem \ref{theo1}, it is only
necessary to obtain the  uniform lower bounds in Proposition \ref{QERh}. 

The proof of Proposition  \ref{QERh}  is based on estimates of variances of restricted matrix elements with respect to logarithmic scale pseudo-differential operators
on $H$,  and then on extraction of density
one subsequences by feeding variance bounds into Chebyshev inequalities. 
The variance bounds are obtained by applying Rellich identities as in \cite{ctz} between variance sums on $M$ and on $H$.We then use the small scale global variance
bounds of \cite{HeR,H}. We begin by defining the variance sums on $M$ and on $H$,
and then review the results of \cite{HeR, H} before going on to the proof of
Proposition \ref{QERh}.

\subsection{Variance sums on $M$ and on $H$}

Given a semi-classical  symbol $a$ on $T^* M$ and its semi-classical Weyl quantization $Op_h(a)= a^w$ on $L^2(M)$, we define the variance sum on $M$ to be 
\begin{equation} \label{VARSUMS}\begin{array}{l}
V_2(\hbar, a): = 
\hbar^{d-1} \sum_{E_j \in [1, 1 + \hbar]}
\left| \langle Op_{\hbar}(a) u_j, u_j \rangle - \mu_L(a_0)
\right|^2\\ \\
:= \frac{1}{N(\lambda)} \sum_{j: \sqrt{\lambda} \in [\lambda, \lambda + 1]} 
\left| \langle Op(a) u_j, u_j \rangle - \mu_L(a_0) \right|^2 \end{array}
\end{equation}
where $d\mu_L$ is normalized Liouville measure and $a_0$ is the principal symbol of $a$. For background on semi-classical symbols and pseudo-differential operators
we refer to \cite{ctz,Zw}.

The restricted variance sums on $H$ have a somewhat different from from \eqref{VARSUMS}. First, in place of the Liouville integral of $a$ one has the restricted state,
\begin{equation} \label{HSATE} \omega_H(a):  =  \int_{B^* H} a(y', \xi')  (1 - |\xi'|^2)^{\half} d y' d \xi'. \end{equation} 
The notation $\omega_H$ is adopted from \cite{HZ,ctz} and we refer there for further discussion.
Secondly, the restriction QER analogue of the matrix element $\langle Op_h(a) u_j, u_j \rangle$ is the matrix element of the Cauchy data $$ CD(u_h |_{H})= (u_h |_H, h d_h u_h |_H)$$ of $u$
on $H$. 
with respect to a semi-classical pseudo-differential operator $Op_h(a)$ on $L^2(H)$:
\begin{equation} \label{CDME}  \begin{array}{lll}  \lll Op_h(a) CD(u_h |_{H}) , CD (u_h |_H) \rrr_{L^2(H)}&:= &
\lll Op_h(a) h \partial_{\nu}  u_h |_{H} , h \partial_{\nu} u_h |_H \rrr_{L^2(H)}\\ &&\\ && + \lll Op_h(a) (1
+ h^2 \Delta_H) u_h |_{H}, u_h |_{H} \rrr_{L^2(H)}. \end{array} \end{equation}
Here, $\Delta_H$ is the Laplacian on $H$ for the metric induced by $g$.
We therefore define restricted variance sums by
\begin{equation} \label{HVARSUMS} \begin{array}{l}
V_{2, H} (\hbar, a): = \hbar^{d-1} \sum_{E_j \in [1, 1 + \hbar]}
\left| \lll Op_h(a) CD(u_h |_{H}) , CD (u_h |_H) \rrr_{L^2(H)}  - \omega_H(a) 
\right|^2\\ \\
:= \frac{1}{N(\lambda)} \sum_{j: \sqrt{\lambda} \in [\lambda, \lambda + 1]} 
\left| \langle Op(a) CD(u_j), CD(u_j )\rangle - \mu_L(a_0) \right|^2 \end{array}
\end{equation}
Here we use the  two different notations that are employed in \cite{JZ,ctz}.

We introduce Fermi normal coordinates $(x_d, x')$ around $H$, in which $H$ is defined by $x_d = 0$.
For a given $k$, we  center coordinates $x_k'$
at the centers $x_k$ of the cover \eqref{COVER}. We then
 consider restricted symbols of the form $f_k(\ell^{-1} x_k')$ where $f_k$ is obtained by transplanting to $B(x_k, C \ell)$ a  fixed  $C^{\infty}$ cutoff function on $\R^d$ which equals $1$ on the ball $B(0,1)$ of radius $1$ and
zero on the complement of twice the ball. That is, $f_k(\ell^{-1} x_k') $ is the pullback of $f_k$ under the chart $x_k'$. We often drop the subscript $k$ on the chart.  As in 
\cite{ctz} (see the discussion around \eqref{ADEF}) we convert the multiplication
operator on $H$ defined by  $f_k(\ell^{-1} x_k') $ 
to  an associated pseudo-differential operator on $M$ given by
\begin{equation} \label{ADEFa}  A_k(x',x^d,h D_x) = \chi ( \frac{x^d}{\epsilon}) \,h
D_{x^{d}}  f_k(\ell^{-1} x'),\end{equation}
where $\chi $ is a $C_0^{\infty}$  cutoff equal to $1$ near $0$, and $\epsilon$ is 
a parameter to be chosen later.

The following Proposition asserts that the restricted Cauchy data matrix elements
\eqref{HVARSUMS}  with respect to $f_k(\ell^{-1} x')$ on $H$ are asymptotic to the globalized matrix elements
on $M$, with a certain dependence on the parameters $\ell, h$:

\begin{prop} \label{RLEMintro} Let $\dim M = d$ and let $H$ be a hypersurface. Let $\{x_k\}$ denote the centers of the cover \eqref{COVER} and let $f = f_k \in C_0^{\infty}(B(x_k, 2 C \ell)$  be defined as above. Let $V_{2}(\hbar, f_k)$ be the
restricted variance sum \eqref{HVARSUMS}, let $A$ be as in \eqref{ADEFa} and let $V_{2, \hbar}(A)$ be as in \eqref{VARSUMS}.
Then,
$$V_{2, H}(\hbar, f_k) = V_{2,\hbar}( A(x',x^d,h D_x)) + {\mathcal
  O} ( \ell^{-2} \hbar) + {\mathcal O} ( \epsilon^{-2} \hbar) + 
   {\mathcal O} ( \epsilon^{-1} \ell^{-1} \hbar),  \;\; \hbar \to 0$$
uniformly in $k$.
\end{prop}

If one picks $\epsilon = \ell$ then the remainder is $\ocal (|\log h|^K h)$. 

We then apply the 
the  global small scale variance estimates on $M$ in  \cite{HeR, H}, which are recalled  
below  in  Proposition \ref{HRVAR} and  Proposition \ref{MAINHAN}. Their results
imply that
\begin{equation} \label{AVAR} V_{2,\hbar}( A(x',x_d,h D_x)) \leq  \frac{C}{|\log \lambda|^{(1 - K \beta)}}.\end{equation} Since the remainders in Proposition \ref{RLEMintro} 
are smaller than the right side of \eqref{AVAR}, we obtain

\begin{cor} \label{VARHEST} With the same notations and assumptions as in Proposition \ref{RLEMintro},
$$V_{2, H}(\hbar, f_k)  \leq \frac{C}{|\log \lambda|^{(1 - K \beta)}}.$$
\end{cor}

The Corollary implies Proposition \ref{QERh}.
The main application of  the result is to even resp. odd eigenfunctions of the negatively curved surface $(M, g)$ with  orientation
reversing involution $\sigma$.   For even eigenfunctions, the $\partial_{\nu} \phi_h$ term is zero, while for odd eigenfunctions
the $(1 + h^2 \Delta) \phi_h$ term is zero. Hence we have,

\begin{cor} \label{RLEMcor} Let $(M, J, \sigma, g)$ be a negatively curved surface as in Theorem \ref{theo1}.  Then for any $\beta > 0$,  the variances for the even eigenfunctions satisfy
\begin{equation}  \label{rellich6} \begin{array}{l}
\hbar^{d-1} \sum_{E_j \in [1, 1 + \hbar]}
\left|  \lll f_k(\ell^{-1} x_k') (1
+ h^2 \Delta_H) \phi_j |_{H}, \phi_j |_{H} \rrr_{L^2(H)}  - \omega_H(a) 
\right|^2  \leq \frac{C}{|\log \lambda_j|^{(1 - K \beta)}},
\end{array} \end{equation}
resp. the variances of the odd eigenfunctions satisfy
\begin{equation}  \label{rellich7} \begin{array}{l}
\hbar^{d-1} \sum_{E_j \in [1, 1 + \hbar]}
\left|   \lll f(\ell^{-1} x_k') h D_\nu \psi_j |_{H} , h D_\nu \psi_j |_H \rrr_{L^2(H)}  - \omega_H(a) 
\right|^2 \leq   \frac{C}{|\log \lambda_j|^{(1 - K \beta)}}
\end{array} \end{equation}
Both remainders are uniform in $k$.
\end{cor}

In  \S \ref{UNIF} we use  Corollary \ref{RLEMcor}  to extract density one subsequences for which one has uniform QER lower bounds
as stated in Corollary \ref{BIGintro},
following the analogous results of \cite{HeR,H}.

\subsection{Review of QE on the logarithmic scale \cite{HeR, H}}

 In this section we   review the results of \cite{HeR, H}.
The first result is Proposition 2.1 of \cite{HeR}. Given $x_0 \in M$ and
$0 < \epsilon < \frac{inj(M,g)}{10})\footnote{$inj(M,g)$ denotes the injectivity radius}, $ define
$$\chi_{x_0, \epsilon} = \chi \left( \frac{||\exp_{x_0}^{-1}(x)||_{x_0}}{\epsilon} \right). $$ Let
$$V_{\lambda}(x_0, \epsilon): = \frac{1}{N(\lambda)} \sum_{j: \lambda_j \leq \lambda} \left| \int_M \chi_{x_0, \epsilon} |\psi_j|^2 dV_g - \int_M \chi_{x_0, \epsilon} \right|^2. $$

\begin{prop}\label{HRVAR}  If $(M, g)$ has negtive sectional curvature, $\beta > 0$ and
$0 < K < \frac{1}{2d}$\footnote{The $\frac{1}{2d}$ constraint on $K$ in the variance estimate is weaker than the constraint $\frac{1}{3d}$ in the uniform QE (and QER) theorems stated in \eqref{ell}. See \S \ref{UNIF}.}
,  then
$$V_{\lambda}(x_0,  (\log \lambda)^{-K}) \leq \frac{C}{|\log \lambda|^{(1 - K \beta)}}. $$
\end{prop}
Note that the symbols $\chi_{x_0, \epsilon} $ are very special in this result,
particularly because they are independent of the $\xi$ variable and thus do not involve
dilations in $\xi$. 

In Section 3.1 of \cite{HeR}, the authors cover $M$ with balls $\{B(x_k, \epsilon)\}_{k =1}^{R(\epsilon)} $ of radius $\epsilon = |\log \lambda|^{-K}$,
$$M \subset \bigcup_{k = 1}^{R(\epsilon)} B(x_k, \epsilon). $$
 The cover has
the property that each point of $M$ is contained in $C_g$ mant of the double
balls $B(x_k, 2 \epsilon)$. The number of such balls satisfies the bounds,
$$c_1 \epsilon^{-d} \leq R(\epsilon) \leq C_2 \epsilon^{-d}. $$

The main QE result of \cite{HeR} gives uniform upper and lower bounds:

\begin{theo}\label{HeRres}With the same assumptions, given any orthonormal basis
of eigenfunctions, there exists a full density subsequence $\Lambda_K$ of
${\mathbb N}$ so that for $j \in \Lambda_K$, and for every
$1 \leq k \leq R(\epsilon), $
$$a_1 \epsilon^{d } \leq \int_{B(x_k, \epsilon)f \cap H} |\psi_j|^2 dS_g
\leq \int_{B(x_k, 50 \epsilon) \cap H} |\psi_j|^2 dS_g \leq a_2 \epsilon^{d },$$
where $a_1, a_2 > 0$ depend only on $\chi, g$. Here,  $\epsilon = |\log \lambda|^{-K}$,
where $K$ is constrained by \eqref{ell}.\end{theo}
A key point is the uniformity of the estimates in the centers $x_k$. We will go over
the argument in \S \ref{UNIF} when we adapt it to the QER setting.

\subsection{Review of X. Han \cite{H}}

X. Han proves a more general small scale QE theorem for semi-classical pseudo-differential
operators with $\delta(\hbar)$- microlocalized symbols, where $\delta(h)$  depends on the way symbols are dilated. 
Theorem 1. 5 of \cite{H} is the main result on small-scale quantum ergodicity. 
It applies to small scale pseudo-differential operators $Op(a)$ where $a$ satisfies symbol estimates of the form
$$\sup |D_{x}^{\alpha} D_{\xi}^{\beta} a| \leq C_{\alpha, \beta} \delta(h)^{- |\alpha| - |\beta|} \langle \xi \rangle^{-|\beta|}, $$
were $\delta(h) = |\log h|^{-\alpha}$ for $\alpha > 0$ satisfying the constraints in  \eqref{deltah}.

Define  the associated variance sums by
$$
V_2(\hbar, a)): = 
\hbar^{d-1} \sum_{E_j \in [1, 1 + \hbar]}
\left| \langle Op_{\hbar}(a) u_j, u_j \rangle - \mu_L(a_{(x_0, \xi_0)}^b)
\right|^2. $$
In Theorem 1.5, Han proves the

\begin{prop}\label{MAINHAN} Let $(M, g)$ be negatively curved. Then
$$V_2(h, a) = O(|\log h|^{-1 + \epsilon}), \;\; \forall \epsilon > 0. $$
\end{prop}
Han states the estimate as $O(|\log h|^{-\beta})$ for $\beta < 1$ when
$\alpha > 0$. When $\alpha = 0$, one can let $\beta = 1$ and then the 
result agrees with \cite{Zel} for non-dilated symbols.

The symbols we will be using have the form
\begin{equation} \label{abb} a^{b_1, b_2}_{x_0}(x, \xi; \hbar) = \delta(\hbar)^{-d} b_1(\frac{x - x_0)}{\delta(\hbar)}) b_2(x, \xi)
\chi(|\xi_g | -1), \end{equation}
where $x_0$ will be chosen to be a center $x_k$ of one of the balls of \eqref{COVER}, where  $b_1 =  f(\ell^{-1} x'), f'(\ell^{-1} x')$ resp. $\chi(\frac{x^d}{\epsilon}) $ (or its
derivative) and where  $b_2 = R_3, R_2, $ resp. $(\xi^d)^2$ is an unscaled classical symbol. See \S \ref{RELLICHSECT} for more details on the notation.
The `base symbol' $b$ will  be one of the
 symbols in terms $I, II, III$ defined below in  \eqref{COMM}. They  only involve rescaling
 in the $x$ variables, while the factors $R_3, R_2, (\xi^d)^2$ are classical un-scaled
 symbols. These are called $\delta(h)$-localized symbols in \cite{H}.
 Special cases have the form,
$$a^b_{x_0} (x, \xi; \hbar) = \delta(\hbar)^{-d} b(\frac{x - x_0)}{\delta(\hbar)})
\chi(|\xi_g | -1). $$
Theorem 1.7 of \cite{H} shows that for $\alpha < \frac{1}{2d}$, $\beta < 1 - 2 \alpha d$,  then 
$$V_2(\hbar, a_{x_0}^b)  = O_b(\delta(h)^{2d} |\log h|^{-\beta}) \;\; \hbar \to 0, $$
uniformly in $x_0$.

Han's uniform comparability result  (Corollary 1.9 of \cite{H})  is analogous to Theorem \ref{HeRres}.

\begin{theo} \label{HANCOR}

 Let $(M, g)$ be negatively curved and of dimension $d$.
Then for any orthonormal basis of eigenfunctions, there exists a density
one subsequence $u_{j_n}$ and a uniform constant $C > 0$  so that for all $x_k$
as in \eqref{COVER},
and with $r(\lambda_{j_n}))$ defined in \eqref{ell},
$$\int_{B(x_k, r(\lambda_{j_n))}} |u_{j_n}|^2  \geq C \; Vol(B(x_k, r_{\lambda_{j_n}})). $$
\end{theo}

There is a slight difference between the integral in Theorem \ref{HANCOR} and
the matrix elements above, namely that we are replacing a smooth symbol by
the characteristic function of a ball. This is by possible by the Portmanteau theorem
for weak* convergence, i.e. the statement that if a sequence $\nu_j \to \nu$ as continuous linear functionals on $C^0(X)$ then $\nu_j(E) \to \nu(E)$ for all
sets $E$ with $\nu(\partial E) = 0$. However in the use of this theorem, rates
of approach to the limit get destroyed, and as in Theorem \ref{HANCOR} one can conlude  an inequality rather than an asymptotic with a remainder.

\subsection{\label{RELLICHSECT} Rellich identity and proof of Proposition
\ref{RLEMintro}}

We now start the proof of Proposition \ref{RLEMintro} and of Proposition \ref{QERh}.
We begin by recalling the Rellich identity as in \cite{ctz}  (based ideas of \cite{GL,Bu})  to convert global QE statements into restricted QER statements. With no loss of generality, we assume  $H$ is a separating hypersurface, so that $H$ is 
the boundary $H = \partial M_+$ of a smooth open submanifold
 $M_+ \subset M$. Let $d S = dS_H$ denote the Riemannian surface
 measure on $H$.

As above, we let   $x =(x^1,...,x^{d-1},x_d)= (x',x^d) $ be Fermi normal coordinates in a small
tubular neighbourhood $H(\epsilon)$ of $H $ defined near a center
$x_k  $ of a ball in the cover \eqref{COVER}. To lighten the notation we drop 
the subscript in $x_k$. Thus, $x = \exp_{x
} x^d \nu_{x'}$ where $x'$ are coordinates on $H$ and $\nu_{x'}$ is the unit
 normal pointing to $M_+$. In these coordinates we can locally define a tubular
 neighborhood of $H$ by
$$H(\epsilon) := \{ (x',x^d) \in U \times {\mathbb R}, \, | x^{d} | < \epsilon \}.$$
Here $U \subset {\mathbb R}^{d-1}$ is a coordinate chart
containing $x_k \in H$ and $\epsilon >0$ is arbitrarily small but
for the moment, fixed. We let $\chi \in C^{\infty}_{0}({\mathbb
R})$ be a cutoff with $\chi(t) = 0$ for $|t| \geq 1$ and $\chi(t)
= 1 $ for $|t| \leq 1/2.$
  In terms of the normal coordinates,
   \begin{equation} \label{DELTA} -h^2\Delta_g = \frac{1}{g(x)} hD_{x^d} g(x) hD_{x^d}   + R(x^d,x',hD_{x'}) \end{equation}
where, $R$ is a second-order $h$-differential operator along $H$
with coefficients that depend on $x_{d}$, and $R(0,x', hD_{x'}) = -h^2 \Delta_H$
is the
induced tangential semiclassical Laplacian on $H$. The main result of this section is,

\begin{lem} \label{RLEM}Let $f \in C^{\infty}(H)$. Let $\epsilon >0$ and define  $\ell$ as in \eqref{ell}.  With the above Fermi normal coordinates around each center $x_k$
of the balls of \eqref{COVER},
\begin{equation}  \label{rellich4} \begin{array}{l}
  \lll f(\ell^{-1} x') h \partial_{\nu}  \phi_h |_{H} , h \partial_{\nu} u_h |_H \rrr_{L^2(H)} + \lll f(\ell^{-1} x') (1
+ h^2 \Delta_H) u_h |_{H}, u_h |_{H} \rrr_{L^2(H)}  \\ \\
 =   \left\langle  \left( \left\{
       (\xi^d)^{2} +
      R(x^d,x',\xi') , \, \chi(\frac{x^d}{\epsilon})  \xi^{d} f(\ell^{-1} x') \right\}  \right)^w u_{h}, \,\,
  u_{h} \right\rangle_{L^2(M_+)}\\ \\+ {\mathcal
  O} ( \ell^{-2} \hbar) + {\mathcal O} ( \epsilon^{-2} \hbar) + 
   {\mathcal O} ( \epsilon^{-1} \ell^{-1} \hbar), 

\end{array} \end{equation}
   where the remainders are uniform in $k$.
\end{lem}

\begin{rem} Compare the statement to that of Proposition \ref{RLEMintro}. \end{rem}

\begin{proof}

Let $A(x, h D_x) \in \Psi_{sc}^{0}(M)$ be an order zero semiclassical pseudodifferential operator on $M$ (see \cite{ctz}). By  Green's formula  we get  the Rellich identity
\begin{align}
\label{rellich}
  \frac{i}{h} \int_{M_+}  & \left( [-h^2 \Delta_{g},  \, A(x, h D_{x}) ] \, u_{h}(x)  \right) \overline{u_{h}(x)} \,  dx \\
= & \int_H \left(h D_{\nu}  \, A(x',x^d,h D_x)
 u_{h}|_{H} \right)  \overline{u_{h}}|_{H}  \,  d S_{H} \notag \\
& +   \int_{H}      \left(\, A(x',x^d,h D_x) \,
  u_{h}|_{H} \right) \overline{h D_{x^d}
    u_{h}}|_{H}   \, dS_{H} .\notag
\end{align}
Here, $D_{x_j} = \frac{1}{i} \frac{\partial}{\partial x^j}$,   $D_{x'}=(D_{x^1},...,D_{x^{d-1}}),$ \,  $D_{x^d} |_H= \frac{1}{i} \partial_{\nu}$ where $\partial_{\nu}$
is the interior unit normal to $M_+$.

We introduce a small parameter $\epsilon$ and  choose
\begin{equation} \label{ADEF}  A(x',x_d,h D_x) = \chi ( \frac{x^d}{\epsilon}) \,h
D_{x^{d}}  f(\ell^{-1} x'). \end{equation}.
Since $A$ is a differential operator of order $1$ and $\Delta_g$ is a differential
operator, the commutator can be explicitly evaluated by elementary calculations. 
The use of pseudo-differential notation is only intended to streamline them.

Since $\chi(0)=1$ it follows that the second term on the right side  of
(\ref{rellich}) is just


\begin{equation} \label{rellich2}
 \lll f(\ell^{-1} x')D_{x^d} u_h |_H, h D_{x^d} u_h |_H \rrr. \end{equation}
The  first term on right hand side   of (\ref{rellich}) equals
\begin{align} \label{rellich3}
  \int_{H} &    h D_{x^d} (\chi(x^d/\epsilon)h D_{x^d} f(\ell^{-1} x') u_{h} )\Big|_{x^d=0}   \overline{ u_{h}}\Big|_{x_d=0}  \,
  dS_{H} \\
= & \int_H \Big( \chi(x^d/\epsilon)f(\ell^{-1} x')(h D_{x^d})^2 u_h +
\frac{h}{i \epsilon} \chi'(x_d/\epsilon)hD_{x^d} f(\ell^{-1} x') u_h
\Big)\Big|_{x^d = 0} \overline{u_h}\Big|_{x^d = 0} d S_H \notag \\
= & \int_H ( \chi(x^d/\epsilon) f(\ell^{-1} x') (1 - R(x^d, x', hD'))
u_h )\Big|_{x^d = 0} \overline{u_h}\Big|_{x^d = 0} d S_H \notag
,
\end{align}
since $\chi'(0) = 0$ and $((hD_{x^d})^2 + R + O(h)) u_h = u_h$ in these
coordinates. Thus, the left side of the stated formula follows from  (\ref{rellich})-(\ref{rellich3}).

It remains to compute  the integral 
$$  \frac{i}{h} \int_{M_+}   \left( [-h^2 \Delta_{g},  \, A(x, h D_{x}) ] \, u_{h}(x)  \right) \overline{u_{h}(x)} \,  dx $$on the left side of  \eqref{rellich}. The commutator 
is  $\hbar$ times a second degree polynomial in  $\hbar D_{x_j}$; thus, the extra factor cancels the factor of $\frac{1}{\hbar}$ outside the integral. 
Using \eqref{DELTA} and \eqref{ADEF}, its  principal symbol is given by
\begin{equation} \label{ps} \begin{array}{lll}\rm{p.s.} ( [-h^2 \Delta_{g},  \, A(x, h D_{x}) ] & = & \left\{
       (\xi^{d})^{2} +
      R(x^d,x',\xi') , \, \chi(\frac{x^d}{\epsilon})  \xi^{d} f(\ell^{-1} x') \right\} 
       \end{array} \end{equation}

The quantization of the principal symbol symbol is simply the naive one taking
$\xi_j \to \hbar D_{x_j}$ in the ordering specified by \eqref{DELTA}. 
The additional  non-principal terms are of one higher order in $\hbar$ but may involve two
derivatives of $\chi(\frac{x^d}{\epsilon})$, two derivatives of $ f(\ell^{-1} x')$
or a product of one mixed derivatives of these functions. This accounts
for the remainder and completes the proof of the Lemma.
\end{proof}

The fact that the remainders are uniform in $k$ is evident from the proof of the variance estimates, and is stated
explicitly in Proposition 2.1 of \cite{HeR} and Theorem 1.7 of  \cite{H}.


This completes the proof of Proposition \ref{RLEMintro}.
The conclusion is 
\begin{cor}\label{VARCORbb}

$$\begin{array}{l} V_{2, H}(\hbar, f_k) = V_{2,\hbar}\left( \hbar,
      \{ (\xi^d)^{2} +
      R(x^d,x',\xi') , \, \chi(\frac{x^d}{\epsilon})  \xi^{d} f(\ell^{-1} x') \}\right) ) \\ \\+ {\mathcal
  O} ( \ell^{-2} \hbar) + {\mathcal O} ( \epsilon^{-2} \hbar) + 
   {\mathcal O} ( \epsilon^{-1} \ell^{-1} \hbar),
 \;\; \hbar \to 0. \end{array}$$
 The remainders are uniform in $k$.

\end{cor}

\subsection{\label{I,II,III} Decomposition  of the gobal variance sums}

The variance sums on the right side of Corollary \ref{VARCORbb} can be simplified
and made more explicit,  because only one term in the Poisson
bracket  of \eqref{ps} dominates. To see this, we observe that
\begin{equation} \label{COMM} \begin{array}{lll}
\left\{
       (\xi^{d})^{2} +
      R(x^d,x',\xi') , \, \chi(\frac{x^d}{\epsilon})  \xi^d f(\ell^{-1} x') \right\} 
 & = &2 \{\xi^d,  \chi(\frac{x^d}{\epsilon}) \}  (\xi^d)^2 f(\ell^{-1} x') \\&&\\
      & +& \{R(x^d, x', \xi'), \xi^d\}  \chi(\frac{x^d}{\epsilon})  f(\ell^{-1} x') \\&&\\ & + &
       \{R(x^d, x', \xi'),  f(\ell^{-1} x')\} \xi^d \chi(\frac{x^d}{\epsilon})\\&&\\
 & = &   \chi(\frac{x^d}{\epsilon})  f(\ell^{-1} x')  R_3(x',x^d,\xi'),\\&&\\&+&  \frac{2}{\epsilon} \chi'(\frac{x_d}{\epsilon})
    (\xi^{d})^{2}  f(\ell^{-1} x') \\ && \\ && +   \ell^{-1} \chi(\frac{x^d}{\epsilon})  f'(\ell^{-1} x')  R_2(x',x^d,\xi')  \\ &&\\
    &= & I + II + III, \end{array}  \end{equation}
where $R_2, R_3$ are zero order symbols which are polynomial of degree $\leq 2$ in $\xi$. The new type of term not encountered
in \cite{ctz} is the second term $II$ in which one takes the Poisson bracket
$\{R(x^d, x', \xi'), f(\ell^{-1} x')\}$, where  $\xi'$ is paired with $x'$. 

 We introduce a further small parameter $\delta$ and   $\chi_2 (t) \in \Ci(\R)$ satisfying
$\chi_2 (t) = 0$ for $t \leq -1/2$, $\chi_2(t) = 1$ for $t \geq
0$, and $\chi_2'(t) >0$ for $-1/2 < t < 0$, and let $\rho$ be a
boundary defining function for $M_+$, i.e. $M_+ =\{\rho \geq 0\}$, $\rho = 0$
on $\partial M_+ = H$ and $d \rho \not= 0$ on $H$.  For instance one may take $\rho = x^d$.  By construction, 
$\chi_2(\rho/\delta) = 1$ on $M_+$ and $= 0$ outside a
$\delta/2$ neighbourhood of $H$ in $M_- = M \backslash M_+$. We choose
$\chi_2, \delta$ so that $\chi_2(\rho/\delta) \chi(x^d/\epsilon) = 1$.

Further, 
 $\chi'(x_n/\epsilon)|_{M_+} = \tchi'(x_n/\epsilon)$ for
a smooth function $\tchi \in \Ci(M)$ satisfying $\tchi = 1$
in a neighbourhood of $M \setminus M_+$ and zero inside
a neighbourhood of $H$. The purpose of the cutoff $\tchi$ is to express matrix elements on $M_+$ as matrix elements on $M$, a manifold
without boundary.

In summary, we now have four  small parameters: $\hbar, \ell, \epsilon, \delta$
 with $\hbar, \ell$ related by \eqref{ell} and cutoffs

\begin{itemize}

\item $\chi(\frac{x_d}{\epsilon})$, where 
 $\chi(t) = 0$ for $|t| \geq 1$ and $\chi(t)
= 1 $ for $|t| \leq 1/2.$
\bigskip

\item $\chi_2(\rho/\delta)$ ($\rho = x^d$), $\chi_2 (t) = 0$ for $t \leq -1/2$, $\chi_2(t) = 1$ for $t \geq
0$, and $\chi_2'(t) >0$ for $-1/2 < t < 0$, ; \bigskip

\item $\tchi (\frac{x^d}{\epsilon})$, $\tchi \in \Ci(M)$ satisfying $\tchi = 1$
in a neighbourhood of $M \setminus M_+$ and zero inside
a neighbourhood of $H$. ;

\end{itemize}
\bigskip

 The variance sums on the right side of Corollary \ref{VARCORbb} are bounded above by the sum of the three sub-sums involving $I, II, III$.
 Proposition \ref{MAINHAN}  applies to all of them. We are interested in the variance sums where $f_k$ is centered at the center
 $x_k$ of a ball in the cover \eqref{COVER}. In addition,  there are the  parameters $\epsilon, \delta$. 
 We introduce some new notation to emphasize this dependence.
 
 \subsubsection{Term I }

We fix $k$ and center the coordinates $x'$ at $x_k$ as discussed above. Then we let $b_1 = 
 \chi(x_d / \epsilon) f(\ell^{-1} x') $ where $x'$ is a local coordinate around $x_k$
 giving $x_k$ the value $0$ and where $b_2 =  R_3(x, \xi') \,).$

We then define (writing $x = (x', x^d)$ as above as Fermi coordinates centered at $x_k$)
$$\begin{array}{l} \rm{Var}_I(\hbar, x_k, \epsilon, \delta):  = 
V_2(\hbar,{\hbar}(\chi(x^d / \epsilon) f(\ell^{-1} x')   R_3(x, \xi')) \\ \\:= 
\hbar^{d-1} \sum_{E_j \in [1, 1 + \hbar]}\\ \\
\left| \langle  ( \,
 \chi(x^d / \epsilon) f(\ell^{-1} x')   R_3(x, \xi') \,)^w
 u_{h}, u_{h} \rangle_{L^{2}(M_+)}  - \int_{S^* M} \chi(x^d / \epsilon) f(\ell^{-1} x')   R_3(x, \xi') \,d \mu_L
\right|^2 \end{array} .$$ 

By the estimate of Proposition \ref{MAINHAN}, 
$$ \rm{Var}_I(\hbar, x_k, \epsilon, \delta): =  O_{b_1, f}( |\log h|^{-1 + \epsilon}).$$
Moreover,
$$\int_{S^* M} \chi(x^d / \epsilon) f(\ell^{-1} x')   R_3(x, \xi') \,d \mu_L
  = \ocal(\epsilon \ell^{d-1}). $$
  Due to the factor of $\epsilon$, these matrix elements will make a negligible
  contribution to the limit $h \to 0$.




\subsubsection{Term $II$}

We define
$$\begin{array}{l}\rm{Var}_{II} (\hbar, x_k, \epsilon, \delta): = 
 \hbar^{d-1} \sum_{E_j \in [1, 1 + \hbar]}\\ \\
\left|  \langle  ( \,
 \frac{2}{\epsilon} \chi'(\frac{x^d}{\epsilon})
     f(\ell^{-1} x')   (\xi^d)^2)^w
 u_{h}, u_{h} \rangle_{L^{2}(M_+)} -  \int_{S^* M}  \frac{2}{\epsilon} \chi'(\frac{x^d}{\epsilon})
     f(\ell^{-1} x')   (\xi^d)^2 \,)d \mu_L\right|^2.
\end{array}$$

Note that

$$\begin{array}{l}
   \left\langle    \left( \frac{1}{\epsilon}
     \chi'(\frac{x^d}{\epsilon}) \, (\xi^d)^2 f(\ell^{-1} x') \right)^w u_{h}, \,\,
   u_{h} \right\rangle_{L^2(M_+)} 
  =  \left\langle  \left( \frac{1}{\epsilon}
     \tchi'(\frac{x^d}{\epsilon}) \, (\xi^d)^2 f(\ell^{-1} x')\right)^w u_{h}, \,\,
   u_{h} \right\rangle_{L^2(M)} \notag \end{array}$$

Again by Proposition \ref{MAINHAN}, 
$$ \rm{Var}_I(\hbar, x_k, \epsilon, \delta): =  O_{b_1, f} |\log h|^{-1 + \epsilon}).$$
But  the presence
of $\frac{1}{\epsilon}$ in $     \frac{1}{\epsilon}
     \chi'(\frac{x^d}{\epsilon})$ ensures that the matrix elements in this variance
     sum have non-trivial limits. Indeed,
     \begin{equation} \label{IISUM} \begin{array}{lll}  \int_{S^* M}  \frac{2}{\epsilon} \chi'(\frac{x^d}{\epsilon})
     f(\ell^{-1} x' )  (\xi^d)^2\,d \mu_L &\simeq & \ell^{d -1}   \int_{B^*H}
     f(y')   (1 - |\xi'|^2)^{\half} \,dy' d \xi'.  \end{array} \end{equation}

\subsubsection{Term $III$}

We define
$$\begin{array}{l}  \rm{Var}_{III}(\hbar, x_k, \epsilon, \delta): = \hbar^{d-1} \sum_{E_j \in [1, 1 + \hbar]}\\ \\
\left|
\ell^{-1} \langle Op_h(\chi(\frac{x^d}{\epsilon})  f'(\ell^{-1} x')  R_2(x',x^d,\xi')) u_h, u_h \rangle_{M_+} 
-  \ell^{-1}  \int_{S^*M} \chi(\frac{x^d}{\epsilon})  f'(\ell^{-1} x')  R_2(x',x^d,\xi') d\mu_L\right|^2
\end{array}$$

Again by  Proposition \ref{MAINHAN}, 
$$ \rm{Var}_I(\hbar, x_k, \epsilon, \delta): =  O_{b_1, f}(|\log h|^{-1 + \epsilon}).$$

The matrix elements in the variance sum  $III$  apriori have non-zero limits  due to the factor $\ell^{-1}$ in $$\ell^{-1} \langle Op_h(\chi(\frac{x^d}{\epsilon})  f'(\ell^{-1} x')  R_2(x',x^d,\xi') )u_h, u_h \rangle_{M_+}. $$
The power of $\ell$ is one higher than in the other terms. However, the factor of $\epsilon$ can be chosen here (and consistently
elsewhere) to be $\ell$ and then the term balances the term $II$. Morever, $\int_{\R^{d-1}} f'(y')d y' = 0$, so the limit vanishes
and the matrix elements in  this term are of order $o(\ell^{d-1})$.

In summary,  the means have the following asymptotics:
$$\begin{pmatrix} **\\*\\*\end{pmatrix}$$

$$\label{rellia} \left\{ \begin{array}{l}  I: \;
\int_{S^* M} \chi(x^d / \epsilon) f(\ell^{-1} x')   R_3(x, \xi') \,d \mu_L 
 \simeq  \epsilon \ell^{d-1} \int_{S^* M} \chi(y^d) f(y')   R_3(\epsilon y^d, \ell y') \,d \mu_L
 \\ \\II: \; \int_{S^* M}  \frac{2}{\epsilon} \chi'(\frac{x^d}{\epsilon})
     f(\ell^{-1} x')   (\xi^d)^2 \,)d \mu_L \simeq 2 \ell^{d-1}  \int_{S^* M}\chi'(y^d)
     f(y')   (\xi^d)^2 \,d \mu_L
 \\ \\III: \; \ell^{-1}  \int_{S^*M} \chi(\frac{x^d}{\epsilon})  f'(\ell^{-1} x')  R_2(x',x^d,\xi') d\mu_L \simeq  \epsilon \ell^{d-1}
 \int_{S^*M} \chi(y^d)  f'(y')  R_2(0, 0,\xi') d\mu_L. \notag
\end{array} \right.$$

\subsection{\label{UNIF} Uniform lower bounds in the centers $x_k$ of the cover: Proof of Proposition \ref{QERh}}

To complete the proof of Proposition \ref{QERh}, we employ the diagonal argument of \cite{HeR} (Section 3.1) or \cite{H} (Proof of Corollary 1.9) to extract
a subsequence of density one for which the lower bound of Proposition \ref{QERh} is valid. The main point is that one is intersecting $|\log h|$ subsequences, and one
needs to use the rate of variance decay to construct  a subsequence of density
one satisfying all $|\log h|$ conditions.
Since the argument is given in  Section 3.1 of \cite{HeR} or in  Section 5.2 of \cite{H},  we only sketch it and explain the modifications necessary for the proof of Proposition \ref{QERh}. 

To clarify the logic of the final argument, we are applying the Chebychev inequality
to the three variance sums $I, II, III$ above.  We use it to extract a subsequence
of indices $j$
of density so that the $j$th summand tends to zero at a certain rate uniformly in $k$.
Since there are $(\log |h|)^K$ values of $k$,  the radii $r(\lambda_j)$
of the shrinking  balls must be slightly larger than would be the case for one fixed $k$. 
More precisely as in \eqref{ell}, $ r(\lambda_j) = (\log \lambda_j)^{- K} \;\rm{ where}\; 0 < K < \frac{1}{3d}$ as opposed to $K < \frac{1}{2d}$ for one fixed $k$. (Compare
Corollary 1.8 and Corollary 1.9 of \cite{H} or the discussion on page 3266).
 Once one has extracted the subsequence, one goes
back to the analysis of the means in $I, II, III$ to see that $II$ contributes the dominant 
asymptotics of the matrix elements. This determines the asymptotics of the
restricted matrix elements by Lemma \ref{RLEM}.

We work separately with the three variance sums $\rm{Var}_I, \rm{Var}_{II}, \rm{Var}_{III}$ of \S \ref{I,II,III}. Proposition \ref{MAINHAN}
applies to all of them, with remainders uniform in $x_k$. 

In the notation of \cite{H} (Step 2, p. 3283), the logarithmic dilation scale
is set as in \eqref{deltah} at $\delta(h) = (|\log h|)^{-\alpha} = r(\lambda_j)$ where $\alpha < \frac{1}{3d}$. Also fix $\beta > 0$. We consider a symbol $a$ of the form \eqref{abb}, or more precisely one of the symbols that arises in $I, II, III$ above, and  define the `exceptional sets'\footnote{The exceptional sets are denoted by $J_{k, K}(h)$ and the generic
sets are denoted $\Lambda_{k, K}$ in \cite{HeR}}
$$\Lambda_{x_k}^b(h): = \{j: E_j \in [1, 1 + h], \;\; |\langle Op_h^b(a_{x_k}^b) u_j, u_j \rangle - \mu_L(a_{x_k}^b) |^2
\geq \delta(h)^{2d} |\log h|^{- 2 \beta}\}. $$
\begin{rem}
In \cite{HeR} the exceptional set is defined by the condition,
$$\left| \int_M \chi_{x_k} |u_j|^2 dV_g - \int_M \chi_{x_k} dV_g \right| \geq |\log h|^{- K \beta}
\int_M \chi_{x_k} dV_g, $$
where $\beta > 0$. The definitions are equivalent because
$ \int_M \chi_{x_k} dV_g \simeq \delta(h)^d$.
\end{rem}

We then apply  Chebyshev's inequality $\rm{Prob}\{X \geq C\} \leq \frac{1}{C} {\bf E} X$  with $X = \langle Op_h^b(a_{x_k}^b) u_j, u_j \rangle - \mu_L(a_{x_k}^b) |^2$ respect to tormalized counting measure of
$E_j \in [1, 1 + h]$ and with $C =  \delta(h)^{2d} |\log h|^{- 2\beta}$. The variance estimate of Proposition \ref{MAINHAN}  says that the expected
 valued of $X$ is $O( |\log h|^{-1 + \epsilon})$.
 It follows that
$$\frac{\# \Lambda_{x_k}^b(h)}{N(h)} \leq (|\log h|)^{2 d \alpha} |\log h|^{2 \beta} |\log h|^{-1 + \epsilon}.$$

\begin{rem} In (7) of \cite{HeR}, with $p = 1$ the authors get $C |\log h|^{K(4 \beta + 2 d) -1}, $
which is equivalent since $\alpha $ of \cite{H}  is $K$ of \cite{HeR} and because $
\beta$ is an arbitrarily small number whose exact definition changes in each occurrence.
\end{rem}

The key point is to obtain uniform upper bounds as $x_k$ varies. Define
$$\Lambda^b(h) = \bigcup_{k = 1}^{N(h)} \Lambda^b_{x_k} (h), \;\; N(h) \leq C |\log h|^{\alpha d}. $$

We further define the `generic sets' 
$$\Gamma^b(h): = \{j: E_j \in [1, 1 + h]\} \backslash \Lambda^b(h). $$

Adding the Chebychev  upper bounds for the $|\log h|^{\alpha d}$ exceptional sets    gives 
$$\frac{\# \Lambda^b(h)}{N(h)} \leq(|\log h|)^{\alpha d} (|\log h|)^{2 d \alpha} |\log h|^{2 \beta} |\log h|^{-1 + \epsilon},$$ hence
$$\frac{\# \Gamma^b(h)}{\# \{j: E_j \in [1, 1+ h]\}}  \geq 1 - \frac{C}{|\log h|^{-  \alpha((4 \beta + 2 d)  + d)} |\log h|},$$
as stated in
(\cite{HeR}, above Lemma 3.1; \cite{H}, p. 3263), and with $|\log h| = \log \lambda$,
The  remainder tends to zero if $- \alpha  ( (4 \beta + 2d)  + d) + 1 > 0.$ Since $\beta$
is arbitrarily small, this requires  $- \alpha 3 d+ 1 > 0$ or $K < \frac{1}{3 d}$. In this case,
$$\frac{\# \Gamma^b(h)}{\# \{j: E_j \in [1, 1+ h]\}} \to 1, \;\;\; h \to 0, $$
thus giving a subsequence of density one.

If $j \in \Gamma^b(h)$ then
$$\left| \langle Op_h(a_{x_k}^b) u_j, u_j \rangle -\mu_L(a_{x_k}^b) \right| \leq C \delta(h)^d |\log h|^{-\beta}$$
uniformly for all $x_k$.

We now let $b$ be one of the symbols occurring in $I, II, III$ and denote by $\Gamma_{I}(h), \Gamma_{II}(h), $ resp. $\Gamma_{III}(h)$
the associated index sets. Recalling that $\delta(h) = (|\log h|)^{-\alpha} = r(\lambda_j)$ where $\alpha < \frac{1}{3d}$, and that  $\beta = 1 -\epsilon$ is a positive
number $< 1$ and arbitrarily close to 1 (cf. Theorem 1.5 of \cite{H}), we have
$$\begin{pmatrix} *\\*\\*\end{pmatrix}$$
\begin{equation}  \left\{ \begin{array}{l} (I)  \; j \in \Gamma_I(h)  \iff
\\  \left| \langle   \,
 (\chi(x^d / \epsilon) f(\ell^{-1} x')   R_3(x, \xi') \,)^w
 u_{h}, u_{h} \rangle_{L^{2}(M_+)}  - \int_{S^* M} \chi(x^d / \epsilon) f(\ell^{-1} x')   R_3(x, \xi') \,d \mu_L
\right| \\\ \ \leq  C  (|\log h|)^{- d \alpha} |\log h|^{-\beta},  \\ \\
(II) \; j \in \Gamma_{II}(h) \iff \\ 
\left|  \langle   \,
 \frac{2}{\epsilon} \chi'(\frac{x^d}{\epsilon})
     f(\ell^{-1} x')   (\xi_d^2)^w
 u_{h}, u_{h} \rangle_{L^{2}(M_+)} - \int_{S^* M}  \frac{2}{\epsilon} \chi'(\frac{x^d}{\epsilon})
     f(\ell^{-1} x')   (\xi^d)^2 \,)d \mu_L\right| \\\ \ \leq  C  (|\log h|)^{- d \alpha}  |\log h|^{-\beta},  
      \\ \\
(III) \; j \in \Gamma_{III}(h) \iff \\ 
\left|
\ell^{-1} \langle Op_h(\chi(\frac{x^d}{\epsilon})  f'(\ell^{-1} x')  R_2(x',x^d,\xi') u_h, u_h \rangle_{M_+} 
-  \ell^{-1}  \int_{S^*M} \chi(\frac{x^d}{\epsilon})  f'(\ell^{-1} x')  R_2(x',x^d,\xi')) d\mu_L\right| \\\ \ \leq  C  (|\log h|)^{-d\alpha}  |\log h|^{-\beta}
 \end{array}\right. .  \end{equation}
 
 All three conditions hold for indices in $\Gamma_I(h) \cap \Gamma_{II}(h) \cap \Gamma_{III}(h)$ and the estimates
 are uniform in $k$. Thus, there exists a subsequence of density one so that the above asymptotics and remainders
 are valid.

\subsection{Implications for restricted matrix elements}

We now prove:

 \begin{lem} I\label{RESTEST}  $j \in \Gamma_I(h) \cap \Gamma_{II}(h) \cap \Gamma_{III}(h),$
 then 
 \begin{equation}  \label{rellich8} \begin{array}{l} \left| \lll f(\ell^{-1} x') CD  \phi_h |_{H} , CD \phi_h |_H \rrr_{L^2(H)}    - \int_{B^*H} f_k(\ell^{-1} x') (1 - |\xi'|^2)^{\half} dx' d\xi'\right
 | \\ \\= O(\ell^d).
\end{array} \end{equation}
\end{lem}

\begin{proof} We use Lemma \ref{RLEM} to convert the restricted matrix
elements into global ones. We then set $\epsilon = \ell$. By Lemma \ref{RESTEST},
the remainders are of small order than the means, as we now verify by evaluating
the means asymptotically.
Throughout we use that  $\ell = |\log h|^{-K}$ with $K < \frac{1}{3d }$, hence the integrals are of order
$|\log h|^{- (d-1) K}  > |\log h|^{- \frac{(d-1)}{3d}}$, thus are  larger than 
the remainder.

We claim that the means have the following asymptotics:
$$\begin{pmatrix} **\\*\\*\end{pmatrix}$$

$$\label{rellia} \left\{ \begin{array}{l}  I: \;
\int_{S^* M} \chi(x^d / \epsilon) f(\ell^{-1} x')   R_3(x, \xi') \,d \mu_L 
 \simeq  \epsilon \ell^{d-1} \int_{S^* M} \chi(y^d) f(y')   R_3(\epsilon y^d, \ell y') \,d \mu_L
 \\ \\II: \; \int_{S^* M}  \frac{2}{\epsilon} \chi'(\frac{x^d}{\epsilon})
     f(\ell^{-1} x')   (\xi^d)^2 \,d \mu_L \simeq 2 \ell^{d-1}  \int_{S^* M}\chi'(y^d)
     f(y')   (\xi^d)^2\,d \mu_L
 \\ \\III: \; \ell^{-1}  \int_{S^*M} \chi(\frac{x^d}{\epsilon})  f'(\ell^{-1} x')  R_2(x',x^d,\xi') d\mu_L \simeq  \epsilon \ell^{d-1}
 \int_{S^*M} \chi(y^d)  f'(y')  R_2(0, 0,\xi') d\mu_L. \notag
\end{array} \right.$$
When we set $\epsilon = \ell$, we find that   the mean in $II$ has the leading
order, thus the restricted matrix element is asymptotic to the mean of $II$.

First we consider the sequence of matrix elments for  term $I$. The mean value of the matrix element is 
$$\int_{S^* M} \chi(x^d / \epsilon) f(\ell^{-1} x')   R_3(x, \xi') \,d \mu_L
  = \ocal(\epsilon \ell^{d-1}). $$
 We may (and will)  set $\epsilon = \ell$ and then the mean is $O(\ell^d)$, while
 the square root of the variance estimate in $\begin{pmatrix} *\\*\\*\end{pmatrix}$ is
 $ (|\log h|)^{- d \alpha} |\log h|^{-\beta} = O(\ell^d |\log h|^{-\beta})$, which is smaller. 
 It follows that with $\epsilon = \ell$,
 \begin{equation} \label{Ielld} \langle  ( \,
 \chi(x^d / \epsilon) f(\ell^{-1} x')   R_3(x, \xi') \,)^w
 u_{h}, u_{h} \rangle_{L^{2}(M_+)}  = O(\ell^d). \end{equation}
 




\subsubsection{Term $II$}

Note that
$$\begin{array}{l}
   \left\langle    \left( \frac{1}{\epsilon}
     \chi'(\frac{x^d}{\epsilon}) \, (\xi^d)^2 f(\ell^{-1} x') \right)^w u_{h}, \,\,
   u_{h} \right\rangle_{L^2(M_+)} 
  =  \left\langle  \left( \frac{1}{\epsilon}
     \tchi'(\frac{x^d}{\epsilon}) \, (\xi^d)^2 f(\ell^{-1} x')\right)^w u_{h}, \,\,
   u_{h} \right\rangle_{L^2(M)} \notag \end{array}$$

   Again with $\epsilon = \ell$, we get 
     \begin{equation} \label{IISUM} \begin{array}{lll}  \int_{S^* M}  \frac{2}{\epsilon} \chi'(\frac{x^d}{\epsilon})
     f(\ell^{-1} x'   (\xi^d)^2 \,d \mu_L &\simeq & \ell^{d -1}   \int_{B^*H}
     f(y')   (1 - |\xi'|^2)^{\half} \,dy' d \xi'.  \end{array} \end{equation}
     Since this is larger than the variance bound, we obtain

\begin{equation} \label{IIelld} 
   \left\langle    \left( \frac{1}{\epsilon}
     \chi'(\frac{x^d}{\epsilon}) \, (\xi^d)^2 f(\ell^{-1} x') \right)^w u_{h}, u_h \right \rangle  \simeq  \ell^{d -1}   \int_{B^*H}
     f(y')   (1 - |\xi'|^2)^{\half} \,dy' d \xi'. \end{equation}

\subsubsection{Term $III$}

Setting $\epsilon = \ell$ and changing variables shows that  $\int_{\R^{d-1}} f'(y')d y' = 0$, so the limit vanishes
and the matrix elements of term $II$ are of order $O(\ell^d)$.

Thus, the full matrix element on $H$ is asymptotic to the mean of term $II$ plus
a remainder of order $\ell^d$.

\end{proof}

Specializing to the special surfaces and applying the Rellich identies, we conclude:

\begin{cor} \label{BIG} Let $(M, J, \sigma, g)$ be a negatively curved surface with isometric
involution. Then
 If $j \in \Gamma_I(h) \cap \Gamma_{II}(h) \cap \Gamma_{III}(h),$
 and if $\phi_j$ is an orthonormal basis of even eigenfunctions, resp. $\{\psi_j\}$
 is an orthonormal basis of odd eigenfunctions, then  for the center  $x_k$ of every ball
 in \eqref{COVER},
 \begin{equation}  \label{rellich8} \begin{array}{l}  \lll f_k(\ell^{-1} x') (1 + h^2 \Delta_H)  \phi_h |_{H} ,  \phi_h |_H \rrr_{L^2(H)}    =   \int_{B^*H}  f_k(\ell^{-1} x') (1 - |\xi'|^2)^{\half} dx' d\xi' 
 | +O(\ell^2).
 \\ \\
\lll f_k(\ell^{-1} x') h D_{x^d}  |_{H} ,  hD_{x^d} \psi_h |_H \rrr_{L^2(H)}    = \int_{B^*H}  f_k(\ell^{-1} x') (1 - |\xi'|^2)^{\half} dx' d\xi' +  O(\ell^2).
\end{array} \end{equation}
uniformly in $k$.

\end{cor}

Next we show that we can eliminate the $\Delta_H$ in $ (1 + h^2 \Delta_H) $ in
at the expense of getting a lower bound. 

\subsection{Completion of the proof of Proposition \ref{QERh}}

Let  $\mu_j$ be a  positive  microlocal lift  of 
 $u_j |_H$ to $B^* H$ in the 
sense that  $d\mu_j $ are positive measures and $\langle Op_h(a) u_j |_H, u_j |_H \rangle \simeq \int_{B^* H} a d\mu_j$\footnote{Sometimes referred to as a Wigner measure}.  In particular,
$$  \lll f_k(\ell^{-1} x') (1 + h^2 \Delta_H)  u_j |_{H} ,  u_j |_H \rrr_{L^2(H)} \simeq \int_{B^*H} (1 - |\xi'|^2) d\mu_j.$$

If $CD(u_j) |_H$ is quantum ergodic on the logarithmic scale at all
centers $x_k$ of \eqref{COVER}, then
then \begin{equation} \label{CDQER} \begin{array}{l}  \lll f_k(\ell^{-1} x') (1 + h^2 \Delta_H)  u_j |_{H} ,  u_j |_H \rrr_{L^2(H)}
+ \int_H f_k(\ell^{-1} x') |\lambda_j^{-\half}  \partial_{\nu} u_j|^2  dS_g
   \\ \\ \simeq   
\int_{B^*H}  f_k(\ell^{-1} x') (1 - |\xi'|^2) d\mu_j. \end{array} \end{equation}
 Since 
$$\begin{array}{lll} \int_H f_{k}(\ell^{-1} x') u_j^2 = \int_{B^* H}  f_{k} (\ell^{-1} x') d\mu_j &\geq &  
\int_{B^*H} f_k (1 - |\xi'|^2) d\mu_j \\ && \\ & \simeq & \lll f_k(\ell^{-1} x') (1 + h^2 \Delta_H)  u_h |_{H} ,  u_h |_H \rrr_{L^2(H)},
 \end{array} $$
 it follows from \eqref{CDQER} that for all $x_k$ and large enough $\lambda_j$,
 \begin{equation}  \label{rellich8} \begin{array}{l}  \int_H f_{k}(\ell^{-1} x') u_j^2 +  \int_H f_k(\ell^{-1} x') |\lambda_j^{-\half}  \partial_{\nu} u_j|^2  dS_g \\ \\\geq     \int_{B^*H} f_{k}(\ell^{-1} x') (1 - |\xi'|^2)^{\half} dx' d\xi' .\end{array} \end{equation}

Proposition \ref{QERh} is an immediate consequence of this inequality. Also Corollary \ref{BIGintro} is
an immediate consequence  of Corollary \ref{BIG} and this
inequality.

\begin{rem} As J. Toth also observed, 
$$\int_H f_k(\ell^{-1} x')| u_j |_H |^2  \geq \langle f_k (I + h^2 \Delta_H) u_j, u_j \rangle_H 
$$
since  $h^2 \Delta_H \leq 0. $ Thus we could also use that
$$\begin{array}{lll} \int_H f_{k}(\ell^{-1} x') u_j^2 = \int_{B^* H}  f_{k} (\ell^{-1} x') d\mu_j &\geq &   \lll f_k(\ell^{-1} x') (1 + h^2 \Delta_H)  u_j |_{H} ,  u_j |_H \rrr_{L^2(H)} \\ && \\ & \simeq & \int_{B^*H} f_k (1 - |\xi'|^2) d\mu_j .
 \end{array}$$ \end{rem}

 \section{\label{KSECT} Proof of Proposition \ref{KUZ}}

The proof is similar to that in \cite{JZ2} but we must keep track of the dependence
of the remainder estimate on $\ell$, i.e. on the number of derivatives of
$f_{\ell}$.

Let $H\Subset M$ be  a smooth $(n-1)$-dimensional orientable hypersurface, and define the Cauchy data of $u_j$ on $H$ as
$$\begin{cases}
\text{Dirichlet data}: & u_j^b= u_j|_H,\\
\text{Neumann data}: & u_j^b= \lambda_j^{-\half} D_{x^d} u_j.
\end{cases}$$

Let $f \in C^{\infty}(H)$  and consider 
$$N(\sqrt{\lambda},  f_k(\ell^{-1} x')) : = \sum_{j: \sqrt{\lambda}_j \leq \sqrt{\lambda}} |\int_H f(\ell^{-1} s) u_j^b |_H (s) dS(s)|^2 . $$

\begin{prop}\label{KUZLEM}
$$N(\sqrt{\lambda},  f_k(\ell^{-1} x')) = \lambda^{\half} \int_H f_{\ell}^2 dS +  O(\lambda^{-\half + \gamma}) \forall \gamma >0. $$
\end{prop}

This is proved by the standard Tauberian method. Let $E(t, x, y)$ denote  the kernel of $\cos t \sqrt{-\Delta}$. We further
denote by $E^b(t, q, q')$ the Dirichlet, resp. Neumann data of the wave kernel on $H$.  First we consider the
cosine transform of 
$$d N(\sqrt{\lambda},  f_k(\ell^{-1} x')): = \sum_{j } |\int_H f_k(\ell^{-1} s) u_j^b |_H (s)dS(s) |^2  \;\ \delta_{\sqrt{\lambda}_j}, $$
given by

\begin{equation}  \begin{array}{lll} S_{ f_k(\ell^{-1} x')}(t): & = & \int_{H} \int_{H}  E^b(t, q, q')  f_k(\ell^{-1} q) f_{k}(\ell^{-1}q') dS(q) dS(q')
\\ &&\\ & = &  \sum_{j} \cos t \sqrt{\lambda_j} \left|\int_{H} f_{k}(\ell^{-1} q)  u_j(q) dS(q) \right|^2 .
\end{array} \end{equation}
We then introduce a smooth cutoff  $\rho \in \scal(\R)$ with $\mbox{supp} \hat{\rho} \subset (-\epsilon, \epsilon)$,
where $\hat{\rho}$ is the Fourier transform of $\rho$, and consider
$$S_{ f_k(\ell^{-1} x')}(\sqrt{\lambda}, \rho) = \int_{\R} \hat{\rho} (t) \;S_{ f_k(\ell^{-1} x')}(t) e^{i t \sqrt{\lambda}} dt. $$

\begin{lem} \label{CONOR} If supp $\hat{\rho}$ is contained in a sufficiently small interval around $0$, with
$\hat{\rho} \equiv 1$ in a smaller interval, then for both Dirichle and Neumann data
$u_j^b$,

\begin{equation}\begin{array}{l}
S_{f_{\ell}}(\lambda, \rho) =  \frac{\pi}{2}\sum_j\rho(\sqrt{\lambda}- \sqrt{\lambda_j}) |\langle u_j^b,  f_k(\ell^{-1} x')\rangle|^2\\ \\
= || f_k(\ell^{-1} x')||_{L^2(H)}^2 + O(\lambda^{-\half + \gamma}), \;\; \forall \gamma > 0, 
\label{eq:BFSSb} \end{array}\end{equation} 

\end{lem}

\begin{proof} 

Until the end, the proof is the same as in \cite{JZ, JZ2}. In \cite{JZ2} we assumed
that the boundary of $M$ was (weakly) concave and here we assume that the curve
$H$ is a geodesic.

 By  wave front set considerations (see \cite{z,HHHZ}), there  exists $\delta_0 > 0$ so that the 
\begin{equation} \label{ep0} \mbox{sing supp} S_{ f_k(\ell^{-1} x')}(t) \cap (- \delta, \delta) = \{0\}. \end{equation}The singular times $t \not= 0$ are the
lengths of   $H$-orthogonal
geodesics, i.e. geodesics which intersect $H$ orthogonally at both endpoints. The singularities at $t \not= 0$ are of lower order than
the singularity at $t = 0$. 

 We now  determine the singularity at $t = 0$ when we introduce the small-scale $f_{\ell}$, using a Hormander style small time  parametrix for the even wave kernel $E(t, x, q)$.There exists an amplitude $A$ so that modulo
smoothing operators,
$$E(t, x, q) \equiv \int_{T_q^* X} A(t, x, q, \xi) \exp \left(i (\langle Exp_q^{-1}(x), \xi \rangle - t |\xi|)\right) d \xi.$$
The amplitude has order zero.
We then take the Cauchy data on $H$.

Changing variables $\xi \to \sqrt{\lambda} \xi$, the trace  may be expressed in the form,
\begin{equation} \label{TRACE} \begin{array}{l} \lambda \int_{\mathbb{R}} \int_{H} \int_{H} \int_{T_q^* X} \hat{\rho}(t) e^{it  \sqrt{\lambda}} \chi_q(\xi) \\ \\A(t, q', q, \lambda \xi) \exp( i \sqrt{\lambda} \left(\langle Exp_q^{-1}(q'), \xi \rangle - t |\xi|\right))\\ \\
 f_{k}(\ell^{-1}q) f_{k}(\ell^{-1}q)  d \xi dt dS(q) dS(q'). \end{array} \end{equation}
 We now compute
the asymptotics by the stationary phase method.
As in \cite{JZ}, we  calculate the expansion by putting
the integral over $T^*_q X $ in polar coordinates,
$$\begin{array}{l}  \int_{\mathbb{R}} \int_0^{\infty} \int_{H} \int_{H} \int_{S_q^* X} \hat{\rho}(t) e^{it  \sqrt{\lambda}} \chi_q(\xi) A(t, q', q) \\ \\\exp( i \sqrt{\lambda} \rho \left(\langle Exp_q^{-1}(q'), \omega \rangle - t  \right)) \\ \\
f_{k}(\ell^{-1} q) f_{k}(\ell^{-1} q')  \rho^{n-1} d \rho d \omega dt dS(q) dS(q'), \end{array} $$
and in these coordinates the phase becomes,
$$ \Psi(q, \rho, t, \omega, q'): =  t +  \rho \langle Exp_q^{-1}(q'), \omega \rangle - t \rho.$$
We get a non-degenerate critical point  in the variables $(t, \rho)$ when $\rho =1, t =  \langle Exp_q^{-1}(q'), \omega \rangle.$

We are mainly interested in the singularity at $t = 0$ and consider that first. If $t = 0$ then $\langle Exp_q^{-1}(q'), \omega \rangle = 0$.
Since $\omega $ is an arbitrary unit co-vector at $q$, this implies $q = q'$. Due to this constraint we need to consider the stationary
phase asymptotics of the full integral, and find that there is a non-degenerate critical manifold given by the diagonal $\rm{diag}(H \times H)
\times S^*_q M$. We write $\omega = (\xi', \xi_d)$ where $\xi_d = \sqrt{1 - |\xi'|^2} \nu$ where $\nu$ is the unit normal to $H$.

When $H$ is totally geodesic (our main application being a closed geodesic of a negatively curved surface), then $Exp_q^{-1} q'$
is tangent to $H$ and thus, $\langle Exp_q^{-1}(q'), \nu \rangle = 0$. Hence we replace the integral over $S^* M$ by an equivalent
integral over the unit coball bundle $B^* H$ of $H$ and replace $\omega = (\xi', \xi_d)$ by $\xi'$. Rescaling,  our integral for fixed $q$ is
$$\begin{array}{l} 
 \frac{1}{\sqrt{\lambda}} \int_0^{\infty} \int_{\R} \int_{B^*H }  e^{i  \sqrt{\lambda} (  t +  \rho \langle Exp_q^{-1}(q'), \xi' \rangle - t \rho)}\\ \\
  \chi_q(\omega) \tilde{A}(t, q', q, \rho \omega)  
f_{k}(\ell^{-1} q) f_{k}(\ell^{-1} q')    \hat{\rho}(t) d \xi'   dS(q') dt d \rho,
\end{array}$$
for another amplitude $\tilde{A}$. We fix $q$ and consider the oscillatory integral in $(q',  \xi', t, \rho)$.

In addition to the non-degenenerate $(\rho, t)$ block there is the $(q', \xi')$ block, which is non-degenerate and has
Hessian determinant one.  Thus, the singularity at $t = 0$ produces the term,
\begin{equation} \eqref{TRACE}\; = \; \int_H A(0, q, q, \nu_q) |f_{k}(\ell^{-1} q)|^2 d S(q) + O(\lambda^{-\half}). \end{equation}

The remainder estimate in Hormander has two contributions. To localize
near the critical point, one needs to integrate by parts, and each time one
gets $O(\ell^{-1} \lambda^{-\half})$. Thus,  when $\ell $  is given by \eqref{ell}, this part of the remainder is $O(\lambda^{-\half + \gamma}$ for
any $\gamma > 0$. Secondly the remainder in the stationary phase expansion
to leading order is $$O(\lambda^{-\half} ||f_{k}(|\ell^{-1} q)||_{C^6}) = O(\lambda^{-\half} \ell^{-6}) = O(\lambda^{-\half + \gamma}), \;\; \forall \gamma > 0. $$
Thus, the remainders are as stated in the Lemma.

\end{proof}
Proposition \ref{KUZLEM} then follows by a standard Tauberian theorem \cite{SV}
(Appendix B).

\subsection{Uniformity for small balls}

As in the logarithmic QER proof, we need to extract a density one subsequence for which the Kuznecov
bounds hold uniformly for all $x_k$. We use the same Chebyshev argument but
based on Propositon \ref{CONOR} rather than on variance sums. The terms
are positive and therefore the same Chebyshev argument gives a subsequence
of density one for which the limits.

We only consider Dirichlet data of Neumann eigenfunctions, since the same
argument is valid for Neumann data of Dirichlet eigenfunctions.
In the notation of \cite{H} (Step 2, p. 3283) we define the `exceptional sets'
$$\Lambda_{x_k}(h): = \{j: E_j \in [1, 1 + h], \;\;|\langle u_j^b, f_{k}(\ell^{-1} x_k') \rangle|^2
\geq (\log \lambda_j)^{1/3} \lambda_j^{-\half} ||f_k(\ell^{-1}(q)||_{L^2(H)}^2.  $$

We again apply  Chebyshev's inequality $\rm{Prob}\{X \geq C\} \leq \frac{1}{C} {\bf E} X$  with $X = \frac{\;|\langle u_j^b, f_{k}(\ell^{-1} x_k') \rangle|^2}{||f_{\ell}||_{L^2(H)}^2}$ with respect to tormalized counting measure of
$E_j \in [1, 1 + h]$ and with $C = \lambda_j^{-\half}   |\log \lambda_j|^{1/3}$.\footnote{Here,
the choice of $1/3$ is rather arbitrary: it could be any number $< \half$. We pick
it to be $ \frac{1}{3}$ to obtain the same constraints on $\alpha$ as for variance sums. } The Kuznecov sum estimate of Proposition \ref{KUZLEM} says that the expected
 valued of $X$ is $O(\lambda^{-1})$.
 It follows that
$$\frac{\# \Lambda_{x_k}^b(h)}{N(h)} \leq (|\log \lambda|)^{-1/3}. $$

Again  define
$$\Lambda(h) = \bigcup_{k = 1}^{N(h)} \Lambda_{x_k} (h), \;\; N(h) \leq C |\log h|^{\alpha d}, $$and
$$\Gamma(h): = \{j: E_j \in [1, 1 + h]\} \backslash \Lambda(h). $$

Adding the 
Chebychev  upper bounds for the $|\log h|^{\alpha d}$ exceptional sets    gives 
$$\frac{\# \Lambda(h)}{N(h)} \leq(|\log h|)^{\alpha d} |\log h|^{-1/3},$$ hence
$$\frac{\# \Gamma(h)}{\# \{j: E_j \in [1, 1_ h]}  \geq 1 -(|\log h|)^{\alpha d - \frac{1}{3}}.$$
Again if $\alpha < \frac{1}{3d}$ we obtain a subsequence of density one.

If $j \in \Gamma(h)$ then
$$|\langle u_j, f_{k}(\ell^{-1} x_k') \rangle|^2
\leq (\log \lambda_j)^{1/3} \lambda_j^{-\half} ||f_{\ell}||_{L^2(H)}^2 $$
uniformly for all $x_k$.

This completes the proof of Proposition \ref{KUZ}.



\section{\label{CONCLUSION} Conclusion of proof}

We now restrict to the surfaces of Theorem \ref{theo1} in dimension $d = 2$,
and conclude the proof along the lines sketched in the Introduction.
We only consider Dirichlet data of even eigenfuntions $\phi_j$; the proof for Neumann data
of odd eigenfunctions is the same. 
The lower bound of Proposition \ref{KUZ} and  \ref{HeRres}  proves that
for a density one subsequence,
$$\int_{B(x_k, \ell) \cap H} |\phi_j|^2 dS_g \geq C_g \ell. $$
Combining with the sup norm estimate
$$|\phi_j|_{L^{\infty}} \leq C_g \frac{\lambda^{\frac{1}{4}}}{\sqrt{\log \lambda}}, $$
we find that along the density one subsequence,
\begin{equation} \label{B} \int_{B(x_k, \ell )\cap H}| \phi_j| \geq \lambda^{-\frac{1}{4}}\sqrt{\log \lambda} \;\ell. 
\end{equation}

But by Proposition \ref{KUZ}, along the density one subsequence
\begin{equation} \label{KUZEST}  |\int_{B(x_k, \ell) \cap H} \phi_j ds|  \leq C_0  \ell \lambda_j^{-\frac{1}{4}} (\log \lambda_j)^{1/3}. \end{equation}
Comparing \eqref{KUZEST} and \eqref{B} shows that 
\begin{equation} \int_{B(x_k, \ell )\cap H } |\phi_j| > \int_{B(x_k, \ell ) \cap H} \phi_j \end{equation} for  all balls in the cover \eqref{COVER}  for sufficiently large $j$ in the density one subsequence. It follows that $\phi_j |_{B(x_k, \ell) \cap H}$ must have
a zero for every $k$.

The rest of the proof of Theorem \ref{theo1} is the same as in \cite{JZ}.

\end{document}